\documentclass[12pt,a4paper,reqno]{amsart}
\usepackage{amsmath}
\usepackage{amssymb}
\usepackage{amsthm}
\usepackage{mathrsfs}
\usepackage[top=1.6in, bottom=1.3in, left=1.3in, right=1.3in]{geometry}
\usepackage[colorlinks]{hyperref}
\usepackage{enumerate}

\author{Carlo Sanna}
\address{Dipartimento di Matematica\\Universit\`a di Torino\\via Carlo Alberto 10\\10123 Torino, Italy}
\email{carlo.sanna.dev@gmail.com}
\urladdr{\url{http://orcid.org/0000-0002-2111-7596}}

\author{Emanuele Tron}
\address{Institut de Math{\'e}matiques de Bordeaux\\Universit{\'e} de Bordeaux\\351 Cours de la Lib{\'e}ration\\33405 Talence, France}
\email{emanuele.tron@u-bordeaux.fr}

\title[Prescribed G.C.D. of $n$ and the $n$th Fibonacci number]
{The density of numbers $n$ having a prescribed G.C.D. with the $n$th Fibonacci number}

\allowdisplaybreaks
\theoremstyle{plain}
\newtheorem{thm}{Theorem}[section]

\newtheorem{lem}[thm]{Lemma}
\theoremstyle{remark}
\newtheorem{rmk}[thm]{Remark}

\def\lcm{\operatorname{lcm}}

\begin{document}

\maketitle

\begin{abstract}
For each positive integer $k$, let $\mathscr{A}_k$ be the set of all positive integers $n$ such that $\gcd(n, F_n) = k$, where $F_n$ denotes the $n$th Fibonacci number.

We prove that the asymptotic density of $\mathscr{A}_k$ exists and is equal to
\begin{equation*}
\sum_{d = 1}^\infty \frac{\mu(d)}{\lcm(dk, z(dk))}
\end{equation*}
where $\mu$ is the M\"obius function and $z(m)$ denotes the least positive integer $n$ such that $m$ divides $F_n$. 
We also give an effective criterion to establish when the asymptotic density of $\mathscr{A}_k$ is zero and we show that this is the case if and only if $\mathscr{A}_k$ is empty.
\end{abstract}

\section{Introduction}

Let $(u_n)_{n \geq 1}$ be a nondegenerate linear recurrence with integral values.
The arithmetic relations between $u_n$ and $n$ are a topic which has attracted the attention of several researchers, especially in recent years. 
For instance, the set of positive integers $n$ such that $u_n$ is divisible by $n$ has been studied by Alba~Gonz\'alez, Luca, Pomerance, and Shparlinski~\cite{MR2928495}, under the mild hypothesis that the characteristic polynomial of $(u_n)_{n \geq 1}$ has only simple roots; and by Andr\'e-Jeannin~\cite{MR1131414}, Luca and Tron~\cite{MR3409327}, Somer~\cite{MR1271392}, and Sanna~\cite{MR3606950}, when $(u_n)_{n \geq 1}$ is a Lucas sequence.
A problem in a sense dual to this is that of understanding when $n$ is coprime to $u_n$. 
In this respect, Sanna~\cite[Theorem 1.1]{San2_preprint} recently proved the following result.
\begin{thm}
The set of positive integers $n$ such that $\gcd(n,u_n)=1$ has a positive asymptotic density, unless $(u_n / n)_{n \geq 1}$ is a linear recurrence.
\end{thm}

In this paper, we focus on the linear recurrence of Fibonacci numbers $(F_n)_{n \geq 1}$, defined as usual by $F_1 = F_2 = 1$ and $F_{n + 2} = F_{n+1} + F_n$ for all integers $n \geq 1$.
For each positive integer $k$, define the set
\begin{equation*}
\mathscr{A}_k := \{n \geq 1 : \gcd(n, F_n) = k\} .
\end{equation*}
Leonetti and Sanna~\cite[Theorems 1.1 and 1.3]{LS_preprint} proved the following:

\begin{thm}
If $\mathscr{B} := \{k \geq 1 : \mathscr{A}_k \neq \varnothing\}$ then its counting function satisfies
\begin{equation*}
\#\mathscr{B}(x) \gg \frac{x}{\log x} ,
\end{equation*}
for all $x \geq 2$.
Furthermore, $\mathscr{B}$ has zero asymptotic density.
\end{thm}

Let $z(m)$ be the \emph{rank of appearance}, or \emph{entry point}, of a positive integer $m$ in the sequence of Fibonacci numbers, that is, the smallest positive integer $n$ such that $m$ divides $F_n$.
It is well known that $z(m)$ exists.
Set also $\ell(m) := \lcm(m, z(m))$.

Our first result establishes the existence of the asymptotic density of $\mathscr{A}_k$ and provides an effective criterion to check whether this asymptotic density is positive.

\begin{thm}\label{thm:exists}
For each positive integer $k$, the asymptotic density of $\mathscr{A}_k$ exists.
Moreover, $\mathbf{d}(\mathscr{A}_k) > 0$ if and only if $\mathscr{A}_k \neq \varnothing$ if and only if $k = \gcd(\ell(k), F_{\ell(k)})$.
\end{thm}

Our second result is an explicit formula for the asymptotic density of $\mathscr{A}_k$.

\begin{thm}\label{thm:main}
For each positive integer $k$, we have
\begin{equation}\label{equ:main}
\mathbf{d}(\mathscr{A}_k) = \sum_{d = 1}^\infty \frac{\mu(d)}{\ell(dk)} ,
\end{equation}
where $\mu$ is the M\"obius function.
\end{thm}

\subsection*{Notation}

Throughout, we reserve the letters $p$ and $q$ for prime numbers.
For a set of positive integers $\mathscr{S}$, we put $\mathscr{S}(x) := \mathscr{S} \cap [1,x]$ for all $x \geq 1$, and we recall that the asymptotic density $\mathbf{d}(\mathscr{S})$ of $\mathscr{S}$ is defined as the limit of the ratio $\#\mathscr{S}(x) / x$, as $x \to +\infty$, whenever this exists. 
As usual, $\mu(n)$, $\tau(n)$, and $P(n)$, denote the M\"obius function, the number of divisors of a positive integer $n$, and the greatest prime factor of an integer $n > 1$, respectively.
We employ the Landau--Bachmann ``Big Oh'' and ``little oh'' notations $O$ and $o$, as well as the associated Vinogradov symbol $\ll$.

\section{Preliminaries}

The next lemma summarizes some basic properties of $\ell$, $z$, and the Fibonacci numbers, which we will implicitly use later without further mention.

\begin{lem}\label{lem:basic}
For all positive integers $m$, $n$ and all prime numbers $p$, we have:
\begin{enumerate}[(i)]
\item $m \mid F_n$ if and only if $z(m) \mid n$.
\item $z(\lcm(m, n)) = \lcm(z(m), z(n))$.
\item $z(p) \mid p - \left(\frac{p}{5}\right)$, where $\left(\frac{p}{5}\right)$ is a Legendre symbol.
\item $\nu_p(F_n) \geq \nu_p(n)$ whenever $z(p) \mid n$.
\item $m \mid \gcd(n, F_n)$ if and only if $\ell(m) \mid n$.
\item $\ell(\lcm(m, n)) = \lcm(\ell(m), \ell(n))$.
\item $\ell(p) = p z(p)$ for $p \neq 5$, while $\ell(5) = 5$.
\end{enumerate}
\end{lem}
\begin{proof}
Facts (i)--(iii) are well-known (see, e.g.,~\cite{MR3141741}).
Fact (iv) follows quickly from the formulas for $\nu_p(F_n)$ given by Lengyel~\cite{MR1337793}.
Finally, (v)--(vii) are easy consequences of (i)--(iii) and the definition of $\ell$.
\end{proof}

Now we state an easy criterion to establish if $\mathscr{A}_k \neq \varnothing$ \cite[Lemma~2.2(iii)]{LS_preprint}.

\begin{lem}\label{lem:Aknonempty}
$\mathscr{A}_k \neq \varnothing$ if and only if $k = \gcd(\ell(k), F_{\ell(k)})$, for all integers $k \geq 1$.
\end{lem}

If $\mathscr{S}$ is a set of positive integers, we define its \emph{set of nonmultiples} as
\begin{equation*}
\mathscr{N}(\mathscr{S}) := \{n \geq 1 : s \nmid n \text{ for all } s \in \mathscr{S} \} .
\end{equation*}
Sets of nonmultiples, or more precisely their complement \emph{sets of multiples}
\begin{equation*}
\mathscr{M}(\mathscr{S}) := \{n \geq 1 : s \mid n \text{ for some } s \in \mathscr{S} \} ,
\end{equation*}
have been studied by several authors, we refer the reader to~\cite{MR1414678} for a systematic treatment of this topic.
We shall need only the following result.

\begin{lem}\label{lem:thinset}
If $\mathscr{S}$ is a set of positive integers such that
\begin{equation*}
\sum_{s \in \mathscr{S}} \frac1{s} < +\infty ,
\end{equation*}
then $\mathscr{N}(\mathscr{S})$ has an asymptotic density.
Moreover, if $1 \notin \mathscr{S}$ then $\mathbf{d}(\mathscr{N}(\mathscr{S})) > 0$.
\end{lem}
\begin{proof}
The part about the existence of $\mathbf{d}(\mathscr{N}(\mathscr{S}))$ is due to Erd{\H o}s~\cite{MR1574879}, while the second assertion follows easily from the inequality
\begin{equation*}
\mathbf{d}(\mathscr{N}(\mathscr{S})) \geq \prod_{s \in \mathscr{S}} \left(1 - \frac1{s}\right)
\end{equation*}
proved by Heilbronn~\cite{Hei37} and Rohrbach~\cite{MR1581555}.
\end{proof}

For any $\gamma > 0$, let us define
\begin{equation*}
\mathscr{Q}_\gamma := \{p : z(p) \leq p^{\gamma}\} .
\end{equation*}
The following is a well-known lemma, which belongs to the folklore.

\begin{lem}\label{lem:Qgamma}
For all $x,\gamma > 0$, we have $\#\mathscr{Q}_\gamma(x) \ll x^{2 \gamma}$.
\end{lem}
\begin{proof}
It is enough noting that
\begin{equation*}
2^{\#\mathscr{Q}_\gamma(x)} \leq \prod_{p \in \mathscr{Q}_\gamma(x)} p \mid \prod_{n \leq x^{\gamma}} F_n \leq 2^{\sum_{n \leq x^{\gamma}} n} = 2^{O(x^{2\gamma})} ,
\end{equation*}
where we employed the inequality $F_n \leq 2^n$, valid for all positive integers $n$.
\end{proof}

\section{Proof of Theorem~\ref{thm:exists}}

We begin by showing that $\mathscr{A}_k$ is a scaled set of nonmultiples.

\begin{lem}\label{lem:AandL}
For each positive integer $k$ such that $\mathscr{A}_k \neq \varnothing$, we have 
\begin{equation*}
\mathscr{A}_k = \left\{\ell(k)m : m \in \mathscr{N}(\mathscr{L}_k) \right\} ,
\end{equation*}
where
\begin{equation*}
\mathscr{L}_k := \{p : p \mid k\} \cup \{\ell(kp) / \ell(k) : p \nmid k\} .
\end{equation*}
\end{lem}
\begin{proof}
We know that $n \in \mathscr{A}_k$ implies $\ell(k) \mid n$. 
Hence, it is enough to prove that $\ell(k)m \in \mathscr{A}_k$, for some positive integer $m$, if and only if $m \in \mathscr{N}(\mathscr{L}_k)$.

Clearly, $\ell(k)m \in \mathscr{A}_k$ for some positive integer $m$, if and only if
\begin{equation}\label{equ:vpgcd0}
\nu_p(\gcd(\ell(k)m, F_{\ell(k)m})) = \nu_p(k)
\end{equation}
for all prime numbers $p$.

Let $p$ be a prime number dividing $k$.
Then, for all positive integer $m$, we have $z(p) \mid \ell(k)m$ and consequently $\nu_p(F_{\ell(k)m}) \geq \nu_p(\ell(k)m)$, so that
\begin{equation}\label{equ:vpgcd1}
\nu_p(\gcd(\ell(k)m, F_{\ell(k)m})) = \nu_p(\ell(k)m) = \nu_p(\ell(k)) + \nu_p(m) .
\end{equation}
In particular, recalling that $k = \gcd(\ell(k), F_{\ell(k)})$ since $\mathscr{A}_k \neq \varnothing$ and thanks to Lemma~\ref{lem:Aknonempty}, for $m = 1$ we get
\begin{equation*}
\nu_p(k) = \nu_p(\gcd(\ell(k), F_{\ell(k)})) = \nu_p(\ell(k)) ,
\end{equation*}
which together with (\ref{equ:vpgcd1}) gives
\begin{equation}\label{equ:vpgcd2}
\nu_p(\gcd(\ell(k)m, F_{\ell(k)m})) = \nu_p(k) + \nu_p(m) .
\end{equation}
Therefore, (\ref{equ:vpgcd0}) holds if and only if $p \nmid m$.

Now let $p$ be a prime number not dividing $k$.
Then (\ref{equ:vpgcd0}) holds if and only if 
\begin{equation*}
p \nmid \gcd(\ell(k)m, F_{\ell(k)m}) .
\end{equation*}
That is, $\ell(p) \nmid \ell(k) m$, which in turn is equivalent to
\begin{equation*}
\frac{\ell(kp)}{\ell(k)} = \frac{\lcm(\ell(k), \ell(p))}{\ell(k)} \nmid m ,
\end{equation*}
since $p$ and $k$ are relatively prime.

Summarizing, we have found that $\ell(k)m \in \mathscr{A}_k$, for some positive integer $m$, if and only if $p \nmid m$ for all prime numbers $p$ dividing $k$, and $\ell(kq)/\ell(k) \nmid m$ for all prime numbers $q$ not dividing $k$, that is, $m \in \mathscr{N}(\mathscr{L}_k)$.
\end{proof}

Now we show that the series of the reciprocals of the $\ell(n)$'s converges.
The methods employed are somehow similar to those used to prove the result of~\cite{MR2740727}.
(See also \cite{MR3125150} for a wide generalization of that result.)

\begin{lem}\label{lem:convergence}
The series
\begin{equation*}
\sum_{n = 1}^\infty \frac1{\ell(n)}
\end{equation*}
converges.
\end{lem}
\begin{proof}
Let $n > 1$ be an integer and put $p := P(n)$.
Clearly, $\lcm(n, z(p))$ is divisible by $\ell(p)$.
Hence, we can write $\lcm(n, z(p)) = \ell(p) m$, where $m$ is a positive integer such that $P(m) \leq p + 1$.
Also, if $p$ and $\lcm(n, z(p))$ are known then $n$ can be chosen in at most $\tau(z(p))$ ways.
Therefore,
\begin{equation*}
\sum_{n = 2}^\infty \frac1{\ell(n)} \leq \sum_{n = 2}^\infty \frac1{\lcm(n, z(P(n)))} \ll \sum_{p} \frac{\tau(z(p))}{p z(p)}\sum_{P(m) \leq p + 1} \frac1{m} ,
\end{equation*}
where we also used the fact that $\ell(p) \gg p z(p)$ for each prime number $p$.
By Mertens' formula~\cite[Chapter~I.1, Theorem~11]{MR1342300}, we have
\begin{equation*}
\sum_{P(m) \leq p + 1} \frac1{m} \leq \prod_{q \leq p + 1} \left(1 - \frac1{q}\right)^{-1} \ll \log p ,
\end{equation*}
for all prime numbers $p$.
Put $\beta := 3/4$ and $\gamma := 1/3$.
It is well known~\cite[Chapter~I.5, Corollary~1.1]{MR1342300} that $\tau(n) \ll_\varepsilon n^\varepsilon$ for any fixed $\varepsilon > 0$.
Hence, $\tau(z(p))\log p \ll p^{1 - \beta}$ for all prime numbers $p$.
Thus we have found that
\begin{equation}\label{equ:conv1}
\sum_{n = 1}^\infty \frac1{\ell(n)} \ll \sum_{p} \frac{\tau(z(p)) \log p}{p z(p)} \ll \sum_{p} \frac1{p^{\beta} z(p)} .
\end{equation}
On the one hand, by partial summation and by Lemma~\ref{lem:Qgamma}, we have
\begin{equation}\label{equ:conv2}
\sum_{p \in \mathscr{Q}_\gamma} \frac1{p^{\beta} z(p)} \leq \sum_{p \in \mathscr{Q}_\gamma} \frac1{p^{\beta}} = \left. \frac{\#\mathscr{Q}_\gamma(t)}{t^{\beta}}\right|_{t = 2}^{+\infty} + \beta\int_2^{+\infty} \frac{\#\mathscr{Q}_\gamma(t)}{t^{\beta + 1}}\,\mathrm{d}t \ll 1 ,
\end{equation}
since $\beta > 2\gamma$.
On the other hand, by the definition of $\mathscr{Q}_\gamma$, we have
\begin{equation}\label{equ:conv3}
\sum_{p \notin \mathscr{Q}_\gamma} \frac1{p^{\beta} z(p)} < \sum_{p} \frac1{p^{\beta + \gamma}} \ll 1 ,
\end{equation}
since $\beta + \gamma > 1$.
Hence, putting together (\ref{equ:conv1}), (\ref{equ:conv2}), and (\ref{equ:conv3}), we get the claim.
\end{proof}

Now we are ready for the proof of Theorem~\ref{thm:exists}.
If $k$ is a positive integer such that $\mathscr{A}_k = \varnothing$ then, obviously, the asymptotic density of $\mathscr{A}_k$ exists and is equal to zero.
So we can assume $\mathscr{A}_k \neq \varnothing$, which in turn, by Lemma~\ref{lem:Aknonempty}, implies that $k = \gcd(\ell(k), F_{\ell(k)})$.
Thanks to Lemma~\ref{lem:convergence}, we have
\begin{equation*}
\sum_{n \in \mathscr{L}_k} \frac1{n} \ll \sum_{p} \frac1{\ell(kp)} \leq \sum_{p} \frac1{\ell(p)} < +\infty ,
\end{equation*}
while clearly $1 \notin \mathscr{L}_k$.
Hence, Lemma~\ref{lem:thinset} tell us that $\mathscr{N}(\mathscr{L}_k)$ has a positive asymptotic density.
Finally, by Lemma~\ref{lem:AandL} we conclude that the asymptotic density of $\mathscr{A}_k$ exists and it is positive.

\section{Proof of Theorem~\ref{thm:main}}

We begin by introducing a family of sets.
For each positive integer $k$, let $\mathscr{B}_k$ be the set of positive integers $n$ such that:

\begin{enumerate}[(i)]
\item $k \mid \gcd(n, F_n)$;
\item if $p \mid \gcd(n, F_n)$ for some prime number $p$, then $p \mid k$.
\end{enumerate}
The essential part of the proof of Theorem~\ref{thm:main} is the following formula for the asymptotic density of $\mathscr{B}_k$.

\begin{lem}\label{lem:Bkdens}
For all positive integers $k$, the asymptotic density of $\mathscr{B}_k$ exists and
\begin{equation}\label{equ:densBk}
\mathbf{d}(\mathscr{B}_k) = \sum_{(d,k) = 1} \frac{\mu(d)}{\ell(dk)} ,
\end{equation}
where the series is absolutely convergent.
\end{lem}
\begin{proof}
For all positive integers $n$ and $d$, let us define
\begin{equation*}
\varrho(n, d) := \begin{cases} 1 & \text{ if } d \mid F_n, \\ 0 & \text{ if } d \nmid F_n .\end{cases}
\end{equation*}
Note that $\varrho$ is multiplicative in its second argument, that is,
\begin{equation*}
\varrho(n, de) = \varrho(n, d) \varrho(n, e)
\end{equation*}
for all relatively prime positive integers $d$ and $e$, and all positive integers $n$.

It is easy to see that $n \in \mathscr{B}_k$ if and only if $\ell(k) \mid n$ and $\varrho(n, p) = 0$ for all prime numbers $p$ dividing $n$ but not dividing $k$.
Therefore,
\begin{align}\label{equ:count1}
\#\mathscr{B}_k(x) &= \sum_{\substack{n \leq x \\ \ell(k) \,\mid\, n}} \prod_{\substack{p \,\mid\, n \\ p \,\nmid\, k}} (1 - \varrho(n, p)) = \sum_{\substack{n \leq x \\ \ell(k) \,\mid\, n}} \sum_{\substack{d \,\mid\, n \\ (d, k) = 1}} \mu(d) \varrho(n, d) \nonumber \\
&= \sum_{\substack{d \leq x \\ (d, k) = 1}} \mu(d) \sum_{\substack{m \leq x/d \\ \ell(k) \,\mid\, dm}} \varrho(dm, d) ,
\end{align}
for all $x > 0$.
Moreover, given a positive integer $d$ which is relatively prime with $k$, we have that $\varrho(dm, d) = 1$ and $\ell(k) \mid dm$ if and only if $\lcm(z(d), \ell(k)) \mid dm$, which in turn is equivalent to $m$ being divisible by
\begin{equation*}
\frac{\lcm(d, \lcm(z(d), \ell(k)))}{d} = \frac{\lcm(\ell(d), \ell(k))}{d} = \frac{\ell(dk)}{d} ,
\end{equation*}
since $d$ and $k$ are relatively prime.
Hence,
\begin{equation*}
\sum_{\substack{m \leq x/d \\ \ell(k) \,\mid\, dm}} \varrho(dm, d) = \sum_{\substack{m \leq x / d \\ \ell(dk)/d \,\mid\, m}} 1 = \left\lfloor\frac{x}{\ell(dk)}\right\rfloor ,
\end{equation*}
for all $x > 0$, which together with (\ref{equ:count1}) implies that
\begin{equation}\label{equ:count2}
\#\mathscr{B}_k(x) = \sum_{\substack{d \leq x \\ (d, k) = 1}} \mu(d) \left\lfloor \frac{x}{\ell(dk)}\right\rfloor = x \sum_{\substack{d \leq x \\ (d, k) = 1}} \frac{\mu(d)}{\ell(dk)} - R(x) ,
\end{equation}
for all $x > 0$, where
\begin{equation*}
R(x) := \sum_{\substack{d \leq x \\ (d, k) = 1}} \mu(d) \left\{ \frac{x}{\ell(dk)}\right\} .
\end{equation*}

Now, thanks to Lemma~\ref{lem:convergence}, we have
\begin{equation*}
\sum_{(d,k) = 1} \frac{|\mu(d)|}{\ell(dk)} \leq \sum_{d = 1}^{\infty} \frac1{\ell(d)} < +\infty .
\end{equation*}
Hence, the series in (\ref{equ:densBk}) is absolutely convergent.

It remains only to prove that $R(x) = o(x)$ as $x \to +\infty$, and then the desired result follows from (\ref{equ:count2}).
We have
\begin{align*}
|R(x)| &\leq \sum_{d \leq x} |\mu(d)|\left\{\frac{x}{\ell(dk)}\right\} = O\!\left(x^{1/2}\right) + \sum_{x^{1/2} \leq d \leq x} |\mu(d)|\left\{\frac{x}{\ell(dk)}\right\} \\
&\leq O\!\left(x^{1/2}\right) + x \sum_{d \,\geq\, x^{1/2}} \frac1{\ell(d)} = o(x) ,
\end{align*}
as $x \to +\infty$, since by Lemma~\ref{lem:convergence} the last series is the tail of a convergent series and hence converges to $0$ as $x \rightarrow +\infty$.
The proof is complete.
\end{proof}

At this point, by the definition of $\mathscr{B}_k$ and by the inclusion-exclusion principle, it follows easily that
\begin{equation*}
\#\mathscr{A}_k(x) = \sum_{d \,\mid\, k} \mu(d) \, \#\mathscr{B}_{dk}(x) ,
\end{equation*}
for all $x > 0$.
Hence, by Lemma~\ref{lem:Bkdens}, we get
\begin{align}\label{equ:last}
\mathbf{d}(\mathscr{A}_k) &= \sum_{d \,\mid\, k} \mu(d) \, \mathbf{d}(\mathscr{B}_{dk}) = \sum_{d \,\mid\, k} \mu(d) \sum_{(e, dk) = 1} \frac{\mu(e)}{\ell(dek)} \nonumber \\
&= \sum_{d \,\mid\, k} \sum_{(e, k) = 1} \frac{\mu(de)}{\ell(dek)} = \sum_{f = 1}^\infty \frac{\mu(f)}{\ell(fk)} , 
\end{align}
since every squarefree positive integer $f$ can be written in a unique way as $f = de$, where $d$ and $e$ are squarefree positive integers such that $d \mid k$ and $\gcd(e, k) = 1$.
Also note that the rearrangement of the series in (\ref{equ:last}) is justified by absolute convergence.
The proof of Theorem~\ref{thm:main} is complete.

\begin{rmk}
As a consequence of Theorem~\ref{thm:main}, note that if $\mathscr{A}_k=\varnothing$ (or equivalently if $k = \gcd(\ell(k), F_{\ell(k)})$, by Lemma~\ref{lem:Aknonempty}) then the series in (\ref{equ:main}) evaluates to $0$, which is not obvious a priori.
\end{rmk}

\section{Generalization to Lucas sequences}

In order to simplify the exposition, we chose to give our results for the sequence of Fibonacci numbers. 
However, they can be easily generalized to every nondegenerate Lucas sequence.
We recall that a Lucas sequence is an integral linear recurrence $(u_n)_{n \geq 0}$ satifying $u_0 = 0$, $u_1 = 1$, and $u_n = a_1 u_{n - 1} + a_2 u_{n - 2}$, for all integers $n \geq 2$, where $a_1$ and $a_2$ are relatively prime integers; while ``nondegenerate'' means that $a_1 a_2 \neq 0$ and that the ratio of the roots of the characteristic polinomial $f_u(X) := X^2 - a_1 X - a_2$ is not a root of unity.

To prove this generalization, there is just a minor complication that must be handled: The rank of appearance $z_u(m)$ of a positive integer $m$ in the Lucas sequence $(u_n)_{n \geq 0}$, that is, the smallest positive integer $n$ such that $m$ divides $u_n$, exists if and only if $m$ is relatively prime with $a_2$.
Therefore, the arguments involving $z(m)$ must be adapted to $z_u(m)$ considering only the positive integers $m$ which are relatively prime with $a_2$.
Except for that, everything works the same, since $z_u(m)$ and $\ell_u(m) := \lcm(m, z_u(m))$ satisfy the same properties of $z(m)$ and $\ell(m)$.
Note only that Lemma~\ref{lem:basic}(iii) must be replaced by: 
\begin{equation*}
z_u(p) \mid p - (-1)^{p-1}\left(\tfrac{\Delta_u}{p}\right),
\end{equation*}
for all prime numbers $p$ not dividing $a_2$, where $\Delta_u := a_1^2 + 4a_2$ is the discriminant of $f_u(X)$.
Also, the analog of Lemma~\ref{lem:basic}(iv), that is, $\nu_p(u_n) \geq \nu_p(n)$ whenever $z_u(p) \mid n$, can be proved, for example, by using the formula for the $p$-adic valuations of the terms of a Lucas sequence given in~\cite{MR3512829}.

With these changes, the following generalization can be proved.

\begin{thm}
Let $(u_n)_{n \geq 0}$ be a nondegenerate Lucas sequence satisfying the recurrence $u_n = a_1 u_{n - 1} + a_2 u_{n - 2}$ for all integers $n \geq 2$, where $a_1$ and $a_2$ are relatively prime integers.
Furthermore, for each positive integer $k$, define the set
\begin{equation*}
\mathscr{A}_{u,k} := \{n \geq 1 : \gcd(n, u_n) = k\} .
\end{equation*}
Then $\mathscr{A}_{u,k} \neq \varnothing$ if and only if $\gcd(k, a_2) = 1$ and $k = \gcd(\ell_u(k), u_{\ell_u(k)})$.
In such a case, $\mathscr{A}_{u,k}$ has an asymptotic density which is given by
\begin{equation*}
\mathbf{d}(\mathscr{A}_{u,k}) = \sum_{(d, a_2) = 1} \frac{\mu(d)}{\ell_u(d k)} ,
\end{equation*}
where the series is absolutely convergent.
\end{thm}

\subsection*{Acknowledgements}

The authors thank the anonymous referee for carefully reading the paper and for suggesting a much simpler proof of Lemma~\ref{lem:convergence}, instead of our original one which was based on arguments similar to those of~\cite[Theorem~5]{MR1836921}.

\bibliographystyle{amsplain}

\end{document}